\documentclass[12pt, reqno]{amsart}
\usepackage{amsmath, amsthm, amscd, amsfonts, amssymb, graphicx, color}
\usepackage[bookmarksnumbered, colorlinks, plainpages]{hyperref}

\input{mathrsfs.sty}

\newtheorem{theorem}{Theorem}[section]
\newtheorem{lemma}[theorem]{Lemma}

\newtheorem{corollary}[theorem]{Corollary}
\theoremstyle{definition}

\theoremstyle{remark}
\newtheorem{remark}[theorem]{Remark}
\numberwithin{equation}{section}
\begin{document}

\title [Complementary and refined inequalities of Callebaut inequality]{Complementary and refined inequalities of Callebaut inequality for operators}

\author[M. Bakherad, M.S. Moslehian]{Mojtaba Bakherad$^1$ and Mohammad Sal Moslehian$^2$}

\address{$^{1}$ Department of Pure Mathematics, Ferdowsi University of Mashhad, P. O. Box 1159, Mashhad 91775, Iran.}
\email{mojtaba.bakherad@yahoo.com; bakherad@member.ams.org}

\address{$^2$ Department of Pure Mathematics, Center of Excellence in
Analysis on Algebraic Structures (CEAAS), Ferdowsi University of
Mashhad, P. O. Box 1159, Mashhad 91775, Iran}
\email{moslehian@um.ac.ir,
moslehian@member.ams.org}


\begin{abstract}
The Callebaut inequality says that
\begin{align*}
\sum_{ j=1}^n \left(A_j\sharp B_j\right)\leq \left(\sum_{ j=1}^n A_j \sigma B_j\right)\sharp\left(\sum_{ j=1}^n A_j \sigma^{\bot} B_j\right)\leq\left(\sum_{ j=1}^n A_j\right)\sharp \left(\sum_{ j=1}^nB_j\right)\,,
\end{align*}
where $A_j, B_j\,\,(1\leq j\leq n)$ are positive invertible operators and $\sigma$ and $\sigma^\perp$ are an operator mean and its dual in the sense of Kabo and Ando, respectively. In this paper we employ the Mond--Pe\v{c}ari\'c method as well as some operator techniques to establish a complementary inequality to the above one under mild conditions. We also present some refinements of a Callebaut type inequality involving the weighted geometric mean and Hadamard products of Hilbert space operators.
\end{abstract} \maketitle

\section{Introduction and preliminaries}
Let ${\mathbb B}({\mathscr H})$ denote the $C^*$-algebra of all
bounded linear operators on a complex Hilbert space ${\mathscr H}$ with the identity $I$. In the case when ${\rm dim}{\mathscr H}=n$, we identify ${\mathbb B}({\mathscr H})$ with the matrix algebra $\mathbb{M}_n$ of all
$n\times n$ matrices with entries in the complex field.
An operator $A\in{\mathbb B}({\mathscr H})$ is called positive
if $\langle Ax,x\rangle\geq0$ for all $x\in{\mathscr H }$, and we then write $A\geq0$. We write $A>0$ if $A$ is
a positive invertible operator. The set of all positive invertible operators (resp., positive definite for matrices) is denoted by ${\mathbb B}({\mathscr H})_+$ (resp., $\mathcal{P}_n)$. For
self-adjoint operators $A, B\in{\mathbb B}({\mathscr H})$, we say
$B\geq A$ if $B-A\geq0$.\\
It is known that the Hadamard
product can be presented by filtering the
tensor product $A \otimes B$ through a positive linear map. In fact,
 $A\circ B=U^*(A\otimes B)U$,
where $U:{\mathscr H}\to {\mathscr H}\otimes{\mathscr H}$ is the isometry
defined by $Ue_j=e_j\otimes e_j$, where $(e_j)$ is an orthonormal basis of the Hilbert space ${\mathscr H}$; see \cite{paul}. In the case of matrices, one easily observe that the Hadamard product of $A=(a_{ij})$ and
$B=(b_{ij})$ is $A\circ B=(a_{ij}b_{ij})$, a principal
submatrix of the tensor product $A\otimes B=(a_{ij}B)_{1 \leq i,j\leq n}$.\\

Let $f$ be a continuous real valued function defined on an interval $J$. It is called operator monotone if
$A\leq B$ implies $f(A)\leq f(B)$ for all self-adjoint operators $A, B\in {\mathbb B}({\mathscr H})$ with spectra in $J$. It said to be operator convex if $f(\lambda A+(1-\lambda)B) \leq \lambda f(A)+(1-\lambda)f(B)$ for all self-adjoint operators $A, B\in {\mathbb B}({\mathscr H})$ with spectra in $J$ and all $\lambda\in [0,1]$.

The axiomatic theory for operator means of positive invertible operators have been developed by Kubo and Ando \cite{ando}. A binary operation $\sigma$ on ${\mathbb B}({\mathscr H})_+$ is called a connection, if the following conditions are satisfied:
\begin{itemize}
\item [(i)] $A\leq C$ and $B\leq D $ imply
$A\sigma B\leq C\sigma D$;
\item [(ii)] $A_n\downarrow A$ and $ B_n\downarrow B$ imply
$A_n\sigma B_n\downarrow A\sigma B$, where $A_n\downarrow A$ means that $A_1\geq A_2\geq \cdots$ and $A_n\rightarrow A$ as $n\rightarrow\infty$ in the strong operator topology;
\item [(iii)] $T^*(A\sigma B)T\leq (T^*AT)\sigma (T^*BT)\,\,(T\in{\mathbb B}({\mathscr H}))$.
\end{itemize}
There exists an affine order isomorphism between the class of
connections and the class of positive operator monotone functions
$f$ defined on $(0,\infty)$  via
$f(t)I=I\sigma(tI)\hspace{.1cm}(t>0)$. In addition, $A\sigma
B=A^{1\over 2}f(A^{-1\over2}BA^{-1\over2})A^{1\over2}$ for all
 $A, B\in{\mathbb B}({\mathscr H})_+$. The operator monotone function
$f$ is called the representing function of $\sigma$. The dual $\sigma^\perp$ of a connection  $\sigma$ with the representing function $f$ is the connection with the representing function $t/f(t)$. A connection $\sigma$ is a mean if it is normalized, i.e. $I\sigma I=I.$ The function $f_{\sharp_\mu}(t)=t^\mu$ on
$(0,\infty)$ for $\mu\in(0,1)$ gives the operator weighted geometric mean $A\sharp_\mu
B=A^{\frac{1}{2}}\left(A^{\frac{-1}{2}}BA^{\frac{-1}{2}}\right)^{\mu}A^{\frac{1}{2}}$. The case $\mu=1/2$ gives rise to the geometric mean $A\sharp B$.
An operator mean $\sigma$ is symmetric if $A\sigma B=B\sigma A$ for all
 $A, B\in{\mathbb B}({\mathscr H})_+$. For a symmetric operator mean $\sigma$, a parametrized operator mean $\sigma_t,\,\, 0 \leq t\leq 1$ is called an interpolational path for $\sigma$ if it satisfies
\begin{itemize}
\item [(1)] $A\sigma_0 B=A$,
$A\sigma_{1/2}B=A\sigma B$, and $A\sigma_1B=B$;
\item [(2)] $(A\sigma_p B)\sigma(A\sigma_qB)=A\sigma_{\frac{p+q}{2}}B$ for all  $p,q\in[0,1]$;
\item [(3)] The map $t\in[0,1]\mapsto A\sigma_tB$ is norm continuous for each $A$ and $B$.
\end{itemize}
It is easy to see that the set of all $r\in[0, 1]$ satisfying
\begin{align}\label{fujii23}
(A\sigma_p B)\sigma_r(A\sigma_qB)=A\sigma_{rp+(1-r)q}B
 \end{align}
for all $p, q$ is a convex subset of $[0, 1]$ including $0$ and $1$.
The power means
\begin{align*}
Am_rB=A^\frac{1}{2}\left(\frac{1+(A^\frac{-1}{2}BA^\frac{-1}{2})^r}{2}\right)^\frac{1}{r}A^\frac{1}{2}\qquad(r\in[-1,1])
 \end{align*}
are some typical interpolational means. Their interpolational paths are
 \begin{align*}
Am_{r,t}B=A^\frac{1}{2}\left({1-t+t(A^\frac{-1}{2}BA^\frac{-1}{2}})^r\right)^\frac{1}{r}A^\frac{1}{2}\qquad(t\in[0,1]).
 \end{align*}
 In particular, $Am_{1,t}B=A\nabla_t B=(1-t)A+tB, Am_{0,t}B=A\sharp_t B$ and $Am_{-1,t}B=A!_t B=\left((1-t)A^{-1}+tB^{-1}\right)^{-1}$. The representing function $F_{r,t}$ of $m_{r,t}$ is  \begin{align*}
F_{r,t}(x)=1m_{r,t}x=(1-t+tx^r)^\frac{1}{r}\qquad(x>0).
 \end{align*}
Daykin et al. \cite{eli} showed the following refinement of the Cauchy--Schwarz inequality. If $f (\cdot,\cdot)$ and $g(\cdot,\cdot)$ are positive functions with two variables on $(0,\infty)\times(0,\infty)$ such that $f(x, y)g(x, y)=x^2y^2$, $f(\lambda x, \lambda y)=\lambda^2 f(x, y)$ and ${yf(x, 1)\over xf(y, 1)}+{xf(y, 1)\over yf(x, 1)}\leq {x\over y}+{y\over x}$ hold for all positive real numbers $x, y , \lambda$, then inequalities
\begin{align*}
\left(\sum_{j=1}^n x_jy_j\right)^2\leq
\sum_{j=1}^n f(x_j,y_j)\sum_{j=1}^n g(x_j,y_j)
\leq \left(\sum_{j=1}^n x_j^2\right)\left(\sum_{j=1}^ny_j^2\right)
 \end{align*}
hold for all positive real numbers $ x_j, y_j \,\,(1\leq j\leq n)$.
A example of such pair of the functions are $f(x, y)=x^{1+s}y^{1-s}$ and $g(x, y)=x^{1-s}y^{1+s}$. Thus we get the following inequality due to Callebaut \cite{CAL}
\begin{align*}
\left(\sum_{j=1}^n x_jy_j\right)^2\leq
\sum_{j=1}^n x_j^{1+s}y_j^{1-s}\sum_{j=1}^n x_j^{1-s}y_j^{1+s}
\leq \left(\sum_{j=1}^n x_j^2\right)\left(\sum_{j=1}^ny_j^2\right),
 \end{align*}
where $ x_j, y_j \,\,(1\leq j\leq n)$ are positive real numbers and $s\in[0,1]$. This is indeed an extension of the Cauchy--Schwarz inequality.
Another example of such pair of the functions are $f(x, y)=x^{2}+y^{2}$ and $g(x, y)=\frac{x^{2}y^{2}}{x^2+y^2}$. Hence we reach the following Milne inequality \cite{eli}
\begin{align*}
\sum_{j=1}^n \sqrt{x_jy_j}\leq
\sqrt{\sum_{j=1}^n \left(x_j+y_j\right)\sum_{j=1}^n \frac{x_jy_j}{x_j+y_j}}
\leq \sqrt{\sum_{j=1}^n x_j\sum_{j=1}^ny_j},
 \end{align*}
where $ x_j, y_j \,\,(1\leq j\leq n)$ are positive real numbers.\\
There have been obtained several Cauchy--Schwarz type inequalities or Hilbert space operators and matrices; see \cite{ABM, I-V} and references therein.
Wada \cite{wada} gave an
operator version of the Callebaut inequality. Hiai and Zhan established a matrix analog of the Callebaut inequality by considering the convexity of a certain norm function \cite{HZ}. In \cite{caleba} the authors showed another operator version of the Callebaut inequality:
\begin{align}\label{caleba21}
\sum_{ j=1}^n \left(A_j\sharp B_j\right)\leq \left(\sum_{ j=1}^n A_j \sigma B_j\right)\sharp\left(\sum_{ j=1}^n A_j \sigma^{\bot} B_j\right)\leq\left(\sum_{ j=1}^n A_j\right)\sharp \left(\sum_{ j=1}^nB_j\right)\,,
\end{align}
where $A_j, B_j\in{\mathbb B}({\mathscr H})_+\,\,(1\leq j\leq n)$ and $\sigma$ is an operator mean.
They presented
\begin{align}\label{caleba223}
\left(\sum_{ j=1}^nA_j\sigma_s B_j\right)\sharp\left(\sum_{ j=1}^n A_j\sigma_{1-s} B_j\right)\leq \left(\sum_{ j=1}^nA_j\sigma_t B_j\right)\sharp\left(\sum_{ j=1}^n A_j\sigma_{1-t} B_j\right)\,,
\end{align}
where $A_j, B_j\in{\mathbb B}({\mathscr H})_+\,\,(1\leq j\leq n)$, $\sigma_t$ is an interpolational
path for $\sigma$ such that $\sigma_t^\perp=\sigma_{1-t}$, $t\in[0,1]$ and $s$ is a real number between $t$ and $1-t$.\\
They also showed that
\begin{align}\label{34rf}
\sum_{j=1}^n(A_j\sharp B_j)\circ \sum_{j=1}^n(A_j\sharp B_j)\nonumber
&\leq\sum_{j=1}^n(A_j\sharp_{s}B_j)\circ \sum_{j=1}^n(A_j\sharp_{1-s}B_j)\nonumber
\\&\leq \sum_{j=1}^n(A_j\sharp_tB_j)\circ \sum_{j=1}^n(A_j\sharp_{1-t} B_j)\nonumber\\&\leq\left(\sum_{j=1}^nA_j\right)\circ \left(\sum_{j=1}^nB_j\right)\,,
\end{align}
where $A_j, B_j\in\mathcal{P}_n\,\,(1\leq j\leq n)$ and either $1\geq t\geq s>{\frac{1}{2}}$ or $0\leq t\leq s<\frac{1}{2}$.\\
In this paper, we present some reverses of inequalities \eqref{caleba21} and \eqref{caleba223} under some mild conditions and discuss some related problems. In the last section, we obtain a refinement of inequality \eqref{34rf}.
\section{Some reverses of the Callebaut inequality for Hilbert space operators}

In this section, we provide some reverses of operator Callebaut inequality under some mild conditions.
It is known \cite[Theorem 5.7]{abc} that for positive operators $A_j, B_j\in{\mathbb B}({\mathscr H})\,\,(1\leq j\leq n)$ it holds that
\begin{align}\label{mean}
\sum_{j=1}^n A_j\sigma B_j \leq \left(\sum_{j=1}^nA_j\right)\sigma \left(\sum_{j=1}^nB_j\right)\,.
\end{align}
We need a reverse of inequality \eqref{mean}.

There is an effective method for finding inverses of some operator inequalities. It was introduced for investigation of converses of the Jensen inequality associated with convex functions and has been shown that the problem of determining multiple or additive complementary inequalities is reduced to solving a single variable maximization or minimization problem, see \cite{abc} and \cite{MT} and references therein. This method sometimes gives also a unified view to several different operator inequalities and can be applied for the study of the Hadamard product, operator means, positive linear maps and other topics in the framework of operator inequalities; cf. \cite{abc2}. We explain it briefly for the operator Choi-Davis-Jensen inequality. It says that if $f$ is an operator concave function on an interval $J$ and $\Phi: \mathbb{B}(\mathscr{H}) \to \mathbb{B}(\mathscr{K})$ is a unital positive linear map, then $f(\Phi(A))\geq \Phi(f(A))$ for all self-adjoint operators $A$ with spectrum in $J$. We need the next result appeared in \cite[Chapter 2]{abc} in some general forms. We state a sketch of its proof for the reader convenience. Incidentally we explain the essence of the Mond--Pe\v{c}ari\'c method.
\begin{theorem}
Let $f$ be a strictly positive concave function on an interval $[m,M]$ with $0<m<M$ and let $\Phi$ be a unital positive linear map. Then
\begin{eqnarray}\label{mond1}
\gamma \Phi(f(A)) \geq f(\Phi(A))
\end{eqnarray}
for all self-adjoint operators $A \in {\mathbb B}({\mathscr H})$ with spectrum in $[m,M]$, where $\mu_f=\frac{f(M)-f(m)}{M-m}$, $\nu_f=\frac{Mf(m)-mf(M)}{M-m}$ and $\gamma=\max\left\{\frac{f(t)}{\mu_f t+\nu_f}: m\leq t\leq M\right\}$.
\end{theorem}
\begin{proof}
Since $f$ is concave we have $f(t)\geq \mu_f t+\nu_f$ for all $t\in[m,M]$. It follows from the continuous functional calculus that $f(A)\geq \mu_f A+\nu_f$ and so $\Phi(f(A))\geq \mu_f \Phi(A)+\nu_f$ for all self-adjoint operators $A$ with spectrum in $[m,M]$. To prove \eqref{mond1}, it therefore is enough to find a scalar $\gamma$ such that show that $\gamma (\mu_f \Phi(A)+\nu_f) \geq f(\Phi(A))$, or by the functional calculus it is sufficient to show that $\gamma (\mu_f t+\nu_f) \geq f(t)$ for all $t\in [m,M]$. Thus $\gamma$ should be $\max\left\{\frac{f(t)}{\mu_f t+\nu_f}: m\leq t\leq M\right\}$, which can be found by maximizing the one variable function $\frac{f(t)}{\mu_f t+\nu_f}$ by usual calculus computations. One should note that there is no $t\geq m$ such that $\mu_f t+\nu_f=0$.
\end{proof}
In the above theorem, if we put $\Phi(X):=\Psi(A)^{-1/2}\Psi(A^{1/2}XA^{1/2})\Psi(A)^{-1/2}$, where $\Psi$ is an arbitrary unital positive linear map and take $f$ to be the representing function of an operator mean $\sigma$, then we reach the inequality
\begin{eqnarray}\label{mond2}
\max\left\{\frac{f(t)}{\mu_f t+\nu_f}: m\leq t\leq M\right\} \Psi(A\sigma B) \geq \Psi(A)\sigma \Psi(B)
\end{eqnarray}
whenever $0\leq mA\leq B\leq MA$.\\
Finally if we take $\Psi$ in \eqref{mond2} to be the positive linear map defined on the diagonal blocks of operators by $\Psi({\rm diag}(A_1, \cdots, A_n))=\frac{1}{n}\sum_{j=1}^nA_j$, then
\begin{eqnarray}\label{main}
\gamma \sum_{j=1}^n A_j\sigma B_j \geq \left(\sum_{j=1}^nA_j\right)\sigma \left(\sum_{j=1}^nB_j\right)\,\, {\rm ~with~} \gamma=\max\left\{\frac{f(t)}{\mu_f t+\nu_f}: m\leq t\leq M\right\}\,\,
\end{eqnarray}
for any positive operators $0< mA_j\leq B_j\leq MA_j\,\,(1\leq j\leq n)$.
If $\sigma=\sharp_\alpha\,\,(\alpha\in[0,1])$, then we reach the following inequality appeared in \cite{Bour}
\begin{align*}
\frac{\alpha^\alpha (M-m) (Mm^\alpha-mM^\alpha)^{\alpha-1}}{(1-\alpha)^{\alpha-1}(M^\alpha-m^\alpha)^{\alpha}} \sum_{j=1}^n A_j\sharp_\alpha B_j \geq \left(\sum_{j=1}^nA_j\right)\sharp_\alpha \left(\sum_{j=1}^nB_j\right)\,.
\end{align*}
In particular, for $\sigma=\sharp=\sharp_{1/2}$ we have the following result due to Lee \cite{LEE}
\begin{align}\label{mond2345}
 \frac{\sqrt{M}+\sqrt{m}}{2\sqrt[4]{Mm}} \sum_{j=1}^n A_j\sharp B_j \geq \left(\sum_{j=1}^nA_j\right)\sharp \left(\sum_{j=1}^nB_j\right)\,.
\end{align}

We are ready to prove our main result of this section, which gives a reverse of double inequality \eqref{caleba21}.
\begin{theorem}\label{plm1}
 Let $0<m A_j \leq B_j \leq M A_j\,\,(1\leq j\leq n)$ and $\sigma$ be a mean with the representing function $f$. Then
\begin{align}\label{mo-mo333}
 \sqrt{\gamma\zeta}\left[\left(\sum_{ j=1}^n A_j \sigma B_j\right)\sharp\left(\sum_{ j=1}^n A_j \sigma^{\bot} B_j\right)\right]\geq  \left( \sum_{ j=1}^n A_j\right)\sharp \left( \sum_{ j=1}^nB_j \right)
\end{align}
where
\begin{eqnarray*}
\mu_f=\frac{f(M)-f(m)}{M-m},\quad \nu_f=\frac{Mf(m)-mf(M)}{M-m}
\end{eqnarray*}
\begin{eqnarray}\label{zeta}
\gamma=\displaystyle{\max_{m\leq t\leq M}} \frac{f(t)}{\mu_f t+\nu_f}\quad{\rm~and~}\quad\zeta=\displaystyle{\max_{m\leq t\leq M}} \frac{f(M)f(m)t}{f(t)(\nu_f t+Mm\mu_f)}.
\end{eqnarray}
In addition,
\begin{align*}
 \frac{\sqrt{M}+\sqrt{m}}{2\sqrt[4]{Mm}}\sum_{ j=1}^n \left(A_j\sharp B_j\right) \geq \left[\left(\sum_{ j=1}^n A_j \sigma B_j\right)\sharp\left(\sum_{ j=1}^n A_j \sigma^{\bot} B_j\right)\right]\,.
\end{align*}
\end{theorem}
\begin{proof}
 Since $f(t)\sharp f(t)^\bot=\sqrt{f(t)\frac{t}{f(t)}}=\sqrt{t}$, we get
 \begin{align}\label{aremet}
(A\sigma B)\sharp(A\sigma^\bot B)= A\sharp B
\end{align}
for all positive operators $A, B$; cf. \cite{ando}. It follows from \eqref{main} that
\begin{align*}
\gamma\sum_{ j=1}^n (A_j\sigma B_j)
 \geq\left(\sum_{j=1}^n A_j \right)\sigma \left(\sum_{j=1}^n B_j\right)
\end{align*}
and
\begin{align*}
\zeta\sum_{j=1}^n (A_j\sigma^\bot B_j)
 \geq\left(\sum_{j=1}^n A_j \right)\sigma^\bot \left(\sum_{j=1}^n B_j\right)\,,
\end{align*}
where $\gamma$ and $\zeta$ are defined by \eqref{zeta}.
It follows from the property (i) of the mean that
\begin{align*}
 \left(\gamma\sum_{j=1}^n A_j \sigma B_j\right)\sharp\left(\zeta\sum_{j=1}^n A_j \sigma^{\bot} B_j\right) \geq\left( \sum_{j=1}^n A_j\sigma\sum_{ i=1}^n B_j \right)\sharp\left( \sum_{j=1}^n A_j\sigma^\bot\sum_{j=1}^n B_j\right)\,.
\end{align*}
Now equality \eqref{aremet} yields that
\begin{align*}
\sqrt{\gamma\zeta} \left[\left(\sum_{ j=1}^n A_j \sigma B_j\right)\sharp\left(\sum_{ j=1}^n A_j \sigma^{\bot} B_j\right) \right]\geq \left( \sum_{ j=1}^n A_j\right)\sharp \left( \sum_{ j=1}^nB_j \right)\,.
\end{align*}
Finally we have
 \begin{align*}
 \frac{\sqrt{M}+\sqrt{m}}{2\sqrt[4]{Mm}}\sum_{ j=1}^n \left(A_j\sharp B_j\right)&\geq\left(\sum_{ j=1}^n A_j\right)\sharp \left(\sum_{ j=1}^nB_j\right)\qquad\qquad\qquad\qquad(\textrm{by~} \eqref{mond2345})\\&\geq\left(\sum_{ j=1}^n A_j \sigma B_j\right)\sharp\left(\sum_{ j=1}^n A_j \sigma^\perp B_j\right)\,.\qquad\qquad(\textrm{by} \eqref{caleba21}).
\end{align*}
\end{proof}
\begin{remark}
Applying \eqref{mond2345} and \eqref{caleba21} we get the following inequality
\begin{align}\label{mo-mo3335}
 \frac{\sqrt{M}+\sqrt{m}}{2\sqrt[4]{Mm}}\left[\left(\sum_{ j=1}^n A_j \sigma B_j\right)\sharp\left(\sum_{ j=1}^n A_j \sigma^{\bot} B_j\right)\right]\geq  \left( \sum_{ j=1}^n A_j\right)\sharp \left( \sum_{ j=1}^nB_j \right),
\end{align}
 where $0<m A_j \leq B_j \leq M A_j\,\,(1\leq j\leq n)$. Now, if we consider the operator function $f(t)=\frac{1+t}{2}$ corresponding to the arithmetic mean, $M=4$ and $m=1$ in \eqref{zeta}, then we observe that
\begin{eqnarray*}
\gamma=\displaystyle{\max_{1\leq t\leq 4}}\frac{{1+t}}{2(\mu_f t+\nu_f)}=1\neq \frac{10}{9}=\displaystyle{\max_{1\leq t\leq 4}}\frac{2f(4)f(1)t}{{(1+t)}(\nu_f t+4\mu_f)}=\zeta.
\end{eqnarray*}
$\sqrt{\gamma\zeta}=\frac{\sqrt{10}}{3}< \frac{3}{2\sqrt{2}}=\frac{\sqrt{M}+\sqrt{m}}{2\sqrt[4]{Mm}}$.
\end{remark}
Using Theorem \ref{plm1} for the function $f(t)=\frac{1+t}{2}$, due to $\gamma=\displaystyle{\max_{m\leq t\leq M}} \frac{f(t)}{\mu_f t+\nu_f}=1$ and $\zeta=\displaystyle{\max_{m\leq t\leq M}} \frac{f(M)f(m)t}{f(t)(\nu_f t+Mm\mu_f)}=\frac{(1+M)(1+m)}{(1+\sqrt{Mm})^2}$ we obtain the following operator version of the reverse Milne inequality.
\begin{corollary}\label{cor2}
 Let $0<m A_j \leq B_j \leq M A_j\,\,(1\leq j\leq n)$. Then
\begin{align}\label{mos-moj}
 \frac{\sqrt{M}+\sqrt{m}}{2\sqrt[4]{Mm}}\sum_{ j=1}^n \left(A_j\sharp B_j\right)&\geq\left[\left(\sum_{j=1}^n A_j \nabla B_j\right)\sharp\left(\sum_{j=1}^n A_j ! B_j\right)\right] \nonumber\\&\geq \frac{1+\sqrt{mM}}{\sqrt{(1+M)(1+m)}}\left( \sum_{ j=1}^n A_j\right)\sharp \left( \sum_{ j=1}^nB_j \right).
\end{align}
\end{corollary}
Now, we show a reverse of \eqref{caleba223} under some mild conditions. First we need the following lemma.
\begin{lemma}\label{fujii123}
Let $$H_{r,t}(x)=\frac{F_{r,t}(x)}{F_{r,1-t}(x)}\qquad(x>0,r\in[-1,1],0\leq t\leq1).$$ Then for a fixed $r$, $H_{r,t}$ is decreasing  for $t\in[0,\frac{1}{2}]$ and increasing  for  $t\in[\frac{1}{2},1]$.
\end{lemma}
\begin{proof}
The case when $r=0$ is clear. Let $r\in[-1,1]-\{0\}$. It follows from $$\frac{d}{dx}\left(H_{r,t}(x)\right)=\left(\frac{(1-t)+tx^r}{t+(1-t)x^r}\right)^{\frac{1}{r}-1}\frac{x^{r-1}(2t-1)}{(t+(1-t)x^r)^2}$$ that $\frac{d}{dx}\left(H_{r,t}(x)\right)\leq0$ for $t\in[0,\frac{1}{2}]$ and $\frac{d}{dx}\left(H_{r,t}(x)\right)\geq0$ for  $t\in[\frac{1}{2},1]$.
Therefore $H_{r,t}(x)$ is  decreasing  for $t\in[0,\frac{1}{2}]$ and is increasing  for  $t\in[\frac{1}{2},1]$.
\end{proof}
\begin{theorem}\label{wewew1}
Let $0<{m} A_j \leq B_j \leq {M} A_j\,\,(1\leq j\leq n)$, $r\in[-1,1]$ and $t\in[0,1]$.  Then
\begin{align}\label{34es}
 \sqrt{\gamma\zeta}\Big[\Big(\sum_{ j=1}^n \left(A_j m_{r,s} B_j\right)\Big)\sharp\Big(\sum_{ j=1}^n &\left(A_jm_{r,1-s} B_j\right)\Big)\Big]\nonumber\\&\geq\left(\sum_{ j=1}^n A_j m_{r,t} B_j\right)\sharp\left(\sum_{ j=1}^n A_j m_{r,1-t} B_j\right),
\end{align}
where $s=s_0t+(1-s_0)(1-t)$ for some $s_0\in[0,1]$ is any number between $t$ and $1-t$,
\begin{eqnarray*}
\mu_{r,s_0}=\frac{F_{r,s_0}(H_{r,t}(M))-F_{r,s_0}(H_{r,t}(m))}{H_{r,t}(M)-H_{r,t}(m)}\,,
\end{eqnarray*}
\begin{eqnarray*}
\nu_{r,s_0}=\frac{H_{r,t}(M)F_{r,s_0}(H_{r,t}(m))-H_{r,t}(m)F_{r,s_0}(H_{r,t}(M))}{H_{r,t}(M)-H_{r,t}(m)}\,,
\end{eqnarray*}
\begin{eqnarray*}
\gamma=\max\left\{\frac{F_{r,s_0}(x)}{\mu_{r,s_0} x+\nu_{r,s_0}}: x {\rm ~is~between~} H_{r,t}(m){\rm~and~} H_{r,t}(M)\right\}
\end{eqnarray*}
and
{\footnotesize \begin{eqnarray*}
\zeta=\max\left\{\frac{F_{r,s_0}(H_{r,t}(M))F_{r,s_0}(H_{r,t}(m))x}{F_{r,s_0}(x)(\nu_{r,s_0} x+H_{r,t}(M)H_{r,t}(m)\mu_{r,s_0})}: x {\rm~is~between~}H_{r,t}(m){\rm~and~} H_{r,t}(M)\right\}\,.
\end{eqnarray*}}
\end{theorem}
\begin{proof}
Assume that $t\in[\frac{1}{2},1]$. It follows from $0<m A_j \leq B_j \leq M A_j\,\,(1\leq j\leq n)$ that  $m\leq A_j^{-1/2}B_jA_j^{-1/2}\leq M\,\,(1\leq j\leq n)$.
 Using Lemma \ref{fujii123} we have
 \begin{align*}
 H_{r,t}(m)\leq H_{r,t}\left(A_j^{-1/2}B_jA_j^{-1/2}\right)\leq H_{r,t}(M)\qquad(-1\leq r\leq1,\,1\leq j\leq n).
 \end{align*}
  So
{\footnotesize\begin{align*}
H_{r,t}(m)F_{r,1-t}\left(A_j^{-1/2}B_jA_j^{-1/2}\right)\leq F_{r,t}\left(A_j^{-1/2}B_jA_j^{-1/2}\right)\leq H_{r,t}(M)F_{r,1-t}\left(A_j^{-1/2}B_jA_j^{-1/2}\right),
\end{align*}}
where $-1\leq r \leq1$ and $1\leq j\leq n$. Multiplying both sides  by $A^{\frac{1}{2}}$ we reach
\begin{align*}
H_{r,t}(m)\left(A_jm_{r,1-t}B_j\right)\leq  A_jm_{r,t} B_j\leq H_{r,t}(M)\left(A_jm_{r,1-t}B_j\right)\,\,(-1\leq r\leq1,\,1\leq j\leq n).
\end{align*}
 Let $s$ be any number between $1-t$ and $t$. So $s=s_0t+(1-s_0)(1-t)$ for some $s_0\in[0,1]$. Using inequality \eqref{mo-mo333} we get
{\footnotesize \begin{align*}
&\Big(\sum_{ j=1}^n A_j m_{r,t} B_j\Big)\sharp\Big(\sum_{ j=1}^n A_j m_{r,1-t} B_j\Big)\\
&\leq  \sqrt{\gamma\zeta}\left[\left(\sum_{ j=1}^n \left((A_j m_{r,t} B_j)m_{r,s_0}(A_j m_{r,1-t} B_j)\right)\right)\sharp\left(\sum_{ j=1}^n \left((A_j m_{r,t} B_j)m_{r,1-s_0}(A_j m_{r,1-t} B_j)\right)\right)\right]\\
&=\sqrt{\gamma\zeta}\left[\left(\sum_{ j=1}^n \left(A_j m_{r,ts_0+(1-t)({1-s_0})} B_j\right)\right)\sharp\left(\sum_{ j=1}^n \left(A_jm_{r,1-(ts_0+(1-t)(1-s_0))} B_j\right)\right)\right]\\&\qquad\qquad\qquad\qquad\qquad\qquad\qquad\qquad\qquad\qquad\qquad\qquad\qquad\qquad\qquad(\textrm {by} ~\eqref{fujii23})\\&=\sqrt{\gamma\zeta}\Big[\Big(\sum_{ j=1}^n \left(A_j m_{r,s} B_j\right)\Big)\sharp\Big(\sum_{ j=1}^n \left(A_jm_{r,1-s} B_j\right)\Big)\Big].
\end{align*}}
Next, assume that $t\in[0,\frac{1}{2}]$. It follows from $0<m A_j \leq B_j \leq M A_j\,\,(1\leq j\leq n)$ that  $m\leq A_j^{-1/2}B_jA_j^{-1/2}\leq M\,\,(1\leq j\leq n)$.
 Using Lemma \ref{fujii123} we have
 \begin{align*}
 H_{r,t}(M)\leq H_{r,t}\left(A_j^{-1/2}B_jA_j^{-1/2}\right)\leq H_{r,t}(m)\qquad(-1\leq r\leq1,\,1\leq j\leq n).
 \end{align*}
  So
{\footnotesize\begin{align*}
H_{r,t}(M)F_{r,1-t}\left(A_j^{-1/2}B_jA_j^{-1/2}\right)\leq F_{r,t}\left(A_j^{-1/2}B_jA_j^{-1/2}\right)\leq H_{r,t}(m)F_{r,1-t}\left(A_j^{-1/2}B_jA_j^{-1/2}\right),
\end{align*}}
  where $-1\leq r \leq1$ and $1\leq j\leq n$.
 Hence {\footnotesize\begin{align*}
\frac{1}{H_{r,t}(m)}F_{r,t}\left(A_j^{-1/2}B_jA_j^{-1/2}\right)\leq F_{r,1-t}\left(A_j^{-1/2}B_jA_j^{-1/2}\right)\leq \frac{1}{H_{r,t}(M)}F_{r,t}\left(A_j^{-1/2}B_jA_j^{-1/2}\right),
\end{align*}}
where $-1\leq r \leq1$ and $1\leq j\leq n$.
Multiplying both sides  by $A^{\frac{1}{2}}$ we reach
\begin{align*}
\frac{1}{H_{r,t}(m)}\left(A_jm_{r,t}B_j\right)\leq  A_jm_{r,1-t} B_j\leq \frac{1}{H_{r,t}(M)}\left(A_jm_{r,t}B_j\right)\,\,(-1\leq r\leq1,\,1\leq j\leq n).
\end{align*}
 Let $s$ be any number between $t$ and $1-t$. So $s=(1-s_0)t+s_0(1-t)$ for some $s_0\in[0,1]$.
 It follows from
\begin{align}\label{O0}
F_{r,s_0}(x^{-1})=(1-s_0+s_0x^{-r})^\frac{1}{r}=x^{-1}((1-s_0)x^r+s_0)^\frac{1}{r}=\frac{F_{r,1-s_0}(x)}{x}\,\,(x>0)
 \end{align}
 that
\begin{align}\label{01}
\mu_{r,1-s_0}&=\frac{F_{r,1-s_0}(\frac{1}{H_{r,t}(m)})-F_{r,1-s_0}(\frac{1}{H_{r,t}(M)})}{\frac{1}{H_{r,t}(m)}-\frac{1}{H_{r,t}(M)}}=
\frac{\frac{F_{r,s_0}(H_{r,t}(m))}{H_{r,t}(m)}-\frac{F_{r,s_0}(H_{r,t}(M))}
{H_{r,t}(M)}}{\frac{H_{r,t}(M)-H_{r,t}(m)}{H_{r,t}(M)H_{r,t}(m)}}
\nonumber\\&=\frac{H_{r,t}(M)F_{r,s_0}(H_{r,t}(m))-H_{r,t}(m)F_{r,s_0}(H_{r,t}(M))}{H_{r,t}(M)-H_{r,t}(m)}=\nu_{r,s_0}
 \end{align}
 and
 \begin{align}\label{02}
\nu_{r,1-s_0}&=\frac{\frac{1}{H_{r,t}(m)}F_{r,1-s_0}(\frac{1}{H_{r,t}(M)})-\frac{1}{H_{r,t}(M)}F_{r,1-s_0}(\frac{1}{H_{r,t}(m)})}
{\frac{1}{H_{r,t}(m)}-\frac{1}{H_{r,t}(M)}}\nonumber\\&=
\frac{F_{r,s_0}({H_{r,t}(M)})-F_{r,s_0}({H_{r,t}(m)})}{H_{r,t}(M)-H_{r,t}(m)}=\mu_{r,s_0}.
 \end{align}
Therefore
 {\footnotesize \begin{align*}
&\max\left\{\frac{F_{r,1-s_0}(x)}{\mu_{r,1-s_0} x+\nu_{r,1-s_0}}: x {\rm ~is~between~} \frac{1}{H_{r,t}(m)}{\rm~and~} \frac{1}{H_{r,t}(M)}\right\}\\&=\max\left\{\frac{F_{r,1-s_0}(x^{-1})}{\mu_{r,1-s_0} x^{-1}+\nu_{r,1-s_0}}: x {\rm ~is~between~} {H_{r,t}(m)}{\rm~and~} {H_{r,t}(M)}\right\}\\&=\max\left\{\frac{\frac{F_{r,s_0}(x)}{x}}{\nu_{r,s_0} x^{-1}+\mu_{r,s_0}}: x {\rm ~is~between~} {H_{r,t}(m)}{\rm~and~} {H_{r,t}(M)}\right\}\\&\qquad\qquad\qquad\qquad\qquad\qquad\qquad\qquad\qquad\qquad\qquad(\textrm{by}~\eqref{O0}, ~\eqref{01} ~\textrm{and}~\eqref{02} )\\&=\max\left\{\frac{F_{r,s_0}(x)}{\nu_{r,s_0}+\mu_{r,s_0}x}: x {\rm ~is~between~} {H_{r,t}(m)}{\rm~and~} {H_{r,t}(M)}\right\}\\&=\gamma
\end{align*}}
and
{\footnotesize\begin{align*}
&\max\left\{\frac{F_{r,1-s_0}(\frac{1}{H_{r,t}(m)})F_{r,1-s_0}(\frac{1}{H_{r,t}(M)})x}
{F_{r,1-s_0}(x)(\nu_{r,1-s_0} x+\frac{1}{H_{r,t}(M)}\frac{1}{H_{r,t}(m)}\mu_{r,1-s_0})}: x {\rm~is~between~}\frac{1}{H_{r,t}(m)}{\rm~and~} \frac{1}{H_{r,t}(M)}\right\}\\&=\max\left\{\frac{F_{r,1-s_0}(\frac{1}{H_{r,t}(m)})F_{r,1-s_0}(\frac{1}{H_{r,t}(M)})x^{-1}}
{F_{r,1-s_0}(x^{-1})(\nu_{r,1-s_0} x^{-1}+\frac{1}{H_{r,t}(M)}\frac{1}{H_{r,t}(m)}\mu_{r,1-s_0})}: x {\rm~is~between~}{H_{r,t}(m)}{\rm~and~} {H_{r,t}(M)}\right\}\\&=\max\left\{\frac{\frac{F_{r,s_0}(H_{r,t}(m))}{H_{r,t}(m)}\frac{F_{r,s_0}(H_{r,t}(M))}{H_{r,t}(M)}x^{-1}}
{\frac{F_{r,s_0}(x)}{x}(\mu_{r,s_0} x^{-1}+\frac{1}{H_{r,t}(M)}\frac{1}{H_{r,t}(m)}\nu_{r,1-s_0})}: x {\rm~is~between~}{H_{r,t}(m)}{\rm~and~} {H_{r,t}(M)}\right\}\\&\qquad\qquad\qquad\qquad\qquad\qquad\qquad\qquad\qquad\qquad\qquad\qquad\qquad(\textrm{by} ~\eqref{O0},,~\eqref{01} ~\textrm{and}~\eqref{02} )\\&=\max\left\{\frac{F_{r,s_0}(H_{r,t}(M))F_{r,s_0}(H_{r,t}(m))x}{F_{r,s_0}(x)(\nu_{r,s_0} x+H_{r,t}(M)H_{r,t}(m)\mu_{r,s_0})}: x {\rm~is~between~}H_{r,t}(m){\rm~and~} H_{r,t}(M)\right\}\\&=\zeta\,.
\end{align*}}
  Using inequality \eqref{mo-mo333} we get
{\footnotesize \begin{align*}
&\Big(\sum_{ j=1}^n A_j m_{r,t} B_j\Big)\sharp\Big(\sum_{ j=1}^n A_j m_{r,1-t} B_j\Big)\\
&\leq  \sqrt{\gamma\zeta}\left[\left(\sum_{ j=1}^n \left((A_j m_{r,t} B_j)m_{r,1-s_0}(A_j m_{r,1-t} B_j)\right)\right)\sharp\left(\sum_{ j=1}^n \left((A_j m_{r,t} B_j)m_{r,s_0}(A_j m_{r,1-t} B_j)\right)\right)\right]\\
&=\sqrt{\gamma\zeta}\left[\left(\sum_{ j=1}^n \left(A_j m_{r,t(1-s_0)+(1-t){s_0}} B_j\right)\right)\sharp\left(\sum_{ j=1}^n \left(A_jm_{r,1-(t(1-s_0)+(1-t)s_0)} B_j\right)\right)\right]\\&\qquad\qquad\qquad\qquad\qquad\qquad\qquad\qquad\qquad\qquad\qquad\qquad\qquad\qquad\qquad(\textrm {by} ~\eqref{fujii23})\\&=\sqrt{\gamma\zeta}\Big[\Big(\sum_{ j=1}^n \left(A_j m_{r,1-s} B_j\right)\Big)\sharp\Big(\sum_{ j=1}^n \left(A_jm_{r,s} B_j\right)\Big)\Big].
\end{align*}}
\end{proof}
Utilizing Theorem \ref{wewew1} for the special case $r=0$ we get the following result.
\begin{corollary}\label{wewew1}
Let $0<{m} A_j \leq B_j \leq {M} A_j\,\,(1\leq j\leq n)$ and $t\in[0,1]$.  Then
{\footnotesize\begin{align}\label{34es}
&\frac{s_0^{s_0} (M^{2t-1}-m^{2t-1}) (M^{2t-1}m^{s_0{(2t-1)}}-m^{2t-1}
 M^{s_0({2t-1})})^{s_0-1}}{(1-s_0)^{(s_0-1)}(M^{s_0({2t-1})}-m^{s_0({2t-1})})^{s_0}} \Big[\Big(\sum_{ j=1}^n \left(A_j \sharp_{s} B_j\right)\Big)\sharp\Big(\sum_{ j=1}^n \left(A_j\sharp_{1-s} B_j\right)\Big)\Big]\nonumber\\&\quad \geq\left(\sum_{ j=1}^n A_j \sharp_t B_j\right)\sharp\left(\sum_{ j=1}^n A_j \sharp_{1-t} B_j\right),
\end{align}}
where $s=s_0t+(1-s_0)(1-t)$ for some $s_0\in[0,1]$ is any number between $t$ and $1-t$.
In particular, if $t=1$, then
\begin{align}\label{mos-moj}
 \frac{s^s (M-m) (Mm^s-mM^s)^{s-1}}{(1-s)^{s-1}(M^s-m^s)^{s}}&\left[\left(\sum_{j=1}^n A_j \sharp_s B_j\right)\sharp\left(\sum_{j=1}^n A_j \sharp_{1-s} B_j\right)\right] \nonumber\\&\geq \left( \sum_{ j=1}^n A_j\right)\sharp \left( \sum_{ j=1}^nB_j \right).
\end{align}
\end{corollary}
\section{A refienement of the Callebaut inequality}
 In this section, we obtain a refinement of inequality \eqref{34rf} for operators. We need the following lemmas.
\begin{lemma}\label{lemma14}$[$See \cite{mo-mo1}$]$
Let $a, b>0$ and $\nu\not\in[0,1]$. Then
\begin{align*}
(a+b)+2(\nu-1)(\sqrt{a}-\sqrt{b})^2\leq a^{\nu}b^{1-\nu}+b^{\nu}a^{1-\nu}.
\end{align*}
 \end{lemma}
 \begin{proof}
Let $\nu\not\in[0,1]$.
Assume that $f(t)=t^{1-\nu}-\nu+(\nu-1)t\,\,(t\in(0,\infty))$. It is easy to see that
$f(t)$ has a minimum at $t=1$ in the interval $(0,\infty)$. Hence
$f(t)\geq f(1)=0$ for all $t>0$. Assume that $a,b>0$. Letting $t={b\over a}$, we get
\begin{align}\label{1b10}
 \nu a+(1-\nu)b\leq a^{\nu}b^{1-\nu}.
 \end{align}
 Now by inequality \eqref{1b10} we have
\begin{align}\label{mos-b}
\nu a+(1-\nu)b+(\nu-1)(\sqrt{a}-\sqrt{b})^2&=(2-2\nu)\sqrt{ab}+(2\nu-1)a\nonumber\\&\leq(\sqrt{ab})^{2-2\nu}a^{2\nu-1}= a^{\nu}b^{1-\nu}.
\end{align}
Similarly
\begin{align}\label{mos-b1}
\nu b+(1-\nu)a+(\nu-1)(\sqrt{b}-\sqrt{a})^2\leq b^{\nu}a^{1-\nu}.
\end{align}
 Adding inequalities \eqref{mos-b} and \eqref{mos-b1} we get the desired inequality.
\end{proof}
\begin{lemma}
Let $A, B\in{\mathbb B}({\mathscr H})_+$ and either $1\geq t\geq s>{\frac{1}{2}}$ or $0\leq t\leq s<\frac{1}{2}$. Then
\begin{align}\label{tool}
A^{s}\otimes B^{1-s}+A^{1-s}\otimes B^{s} \nonumber&+\left(\frac{t-s}{s-1/2}\right)\left(A^{s}\otimes B^{1-s}+A^{1-s}\otimes B^{s}-2(A^{\frac{1}{2}}\otimes B^{\frac{1}{2}})\right) \nonumber\\&\leq A^{t}\otimes B^{1-t}+A^{1-t}\otimes B^{t}\,.
 \end{align}
\end{lemma}
\begin{proof}
If we put $a^{-1}$ instead of $b$ and $s$ instead of $2\nu-1$, respectively, in Lemma \ref{lemma14} we get
 \begin{align*}
a+a^{-1}+(s-1)(a+a^{-1}-2)\leq a^{s}+a^{-s}\qquad(a>0, s\geq1).
 \end{align*}
Let us fix positive real numbers $\alpha, \beta$ such that $\beta\geq{\alpha}$. Using the functional calculus, if we replace $a$ by $A^{\alpha}\otimes B^{-\alpha}$ and $s$ by $\frac{\beta}{\alpha}$, then we get
\begin{align}\label{apq}
A^{\alpha}\otimes B^{-\alpha}\nonumber&+ A^{-\alpha}\otimes B^{\alpha}\nonumber\\&\qquad+\left(\frac{\beta-\alpha}{\alpha}\right) \left(A^{\alpha}\otimes B^{-\alpha}+A^{-\alpha}\otimes B^{\alpha}-2I\right)\nonumber\\&\leq A^{\beta}\otimes B^{-\beta}+A^{-\beta}\otimes B^{\beta}.
 \end{align}
Multiplying both sides of \eqref{apq} by $A^{\frac{1}{2}}\otimes B^{\frac{1}{2}}$ we reach
\begin{align}\label{frac23}
A^{1+\alpha}\otimes B^{1-\alpha}\nonumber&+A^{1-\alpha}\otimes B^{1+\alpha}\nonumber \qquad\\&+\left(\frac{\beta-\alpha}{\alpha}\right)\left(A^{1+\alpha}\otimes B^{1-\alpha}+A^{1-\alpha}\otimes B^{1+\alpha}-2(A\otimes B)\right)\nonumber \\&\leq A^{1+\beta}\otimes B^{1-\beta}+A^{1-\beta}\otimes B^{1+\beta}\,.
 \end{align}
 Now, if we replace $\alpha,\beta$, $A, B$ by $2s-1, 2t-1$, $A^{\frac{1}{2}}, B^{\frac{1}{2}}$, respectively, in \eqref{frac23}, we obtain
 \begin{align*}
A^{s}\otimes B^{1-s}+A^{1-s}\otimes B^{s} \nonumber&+\left(\frac{t-s}{s-1/2}\right)\left(A^{s}\otimes B^{1-s}+A^{1-s}\otimes B^{s}-2(A^{\frac{1}{2}}\otimes B^{\frac{1}{2}})\right) \nonumber\\&\leq A^{t}\otimes B^{1-t}+A^{1-t}\otimes B^{t}\,.
 \end{align*}
 \end{proof}
We are ready to establish the main result of this section.
\begin{theorem}\label{okmn}
Let $A_j, B_j\in{\mathbb B}({\mathscr H})_+\,\,(1\leq j\leq n)$. Then
\begin{align*}
&\sum_{j=1}^n(A_j\sharp_{s}B_j)\circ \sum_{j=1}^n(A_j\sharp_{1-s}B_j)
\\&\leq\sum_{j=1}^n(A_j\sharp_{s}B_j)\circ \sum_{j=1}^n(A_j\sharp_{1-s}B_j)
\\&\quad+\left(\frac{t-s}{s-1/2}\right)\left(\sum_{j=1}^n(A_j\sharp_{s}B_j)\circ \sum_{j=1}^n(A_j\sharp_{1-s}B_j)
-\sum_{j=1}^n(A_j\sharp B_j)\circ \sum_{j=1}^n(A_j\sharp B_j)\right)
\\&\leq \sum_{j=1}^n(A_j\sharp_{t}B_j)\circ \sum_{j=1}^n(A_j\sharp_{1-t} B_j)\,,
 \end{align*}
 for $1\geq t\geq s>{\frac{1}{2}}$ or $0\leq t\leq s<\frac{1}{2}$.
\end{theorem}
\begin{proof} The first inequality is clear. We prove the second one. Put $C_j=A_j^{-{1\over2}}B_jA_j^{-{1\over2}}\,\,(1\leq j\leq n)$. By inequality \eqref{tool} we get
\begin{align}\label{evs}
C_j^{s}\otimes C_i^{1-s}+C_j^{1-s}\otimes C_i^{s}&+\left(\frac{t-s}{s-1/2}\right)
\left(C_j^{s}\otimes C_i^{1-s}+C_j^{1-s}\otimes C_i^{s}-2\left(C_j^{\frac{1}{2}}\otimes C_i^{\frac{1}{2}}\right)\right)\nonumber\\&\leq C_j^{t}\otimes C_i^{1-t}+C_j^{1-t}\otimes C_i^{t}\qquad(1\leq i,j\leq n).
 \end{align}
Multiplying both sides of \eqref{evs} by $A_j^{\frac{1}{2}}\otimes A_i^{\frac{1}{2}}$ we get

 \begin{align}\label{9090}
(A_j&\sharp_{s} B_j)\otimes (A_i\sharp_{1-s} B_i)+(A_j\sharp_{1-s} B_j)\otimes (A_i\sharp_{s}B_i)+\left(\frac{t-s}{s-1/2}\right)
\Big((A_j\sharp_{s} B_j)\otimes (A_i\sharp_{1-s} B_i)\nonumber\\&\quad+(A_j\sharp_{1-s} B_j)\otimes (A_i\sharp_{s} B_i)-2(A_j\sharp B_j)\otimes (A_i\sharp B_i)\Big)\nonumber\\&\leq
(A_j\sharp_{t}B_j)\otimes (A_i\sharp_{1-t} B_i)+(A_j\sharp_{1-t}B_j)\otimes (A_i\sharp_{t}B_i)\,.
 \end{align}
for all $1\leq i,j\leq n$. Therefore
 \begin{align*}
&\sum_{j=1}^n(A_j\sharp_{s}B_j)\circ \sum_{j=1}^n(A_j\sharp_{1-s}B_j)
\\ & \,+\left(\frac{t-s}{s-1/2}\right)\left(\sum_{j=1}^n(A_j\sharp_{s}B_j)\circ \sum_{j=1}^n(A_j\sharp_{1-s}B_j)
-\left(\sum_{j=1}^nA_j\sharp B_j\right)\circ\left(\sum_{j=1}^n A_j\sharp B_j\right)\right)\\&=
\frac{1}{2}\sum_{i,j=1}^n\Big((A_j\sharp_{s} B_j)\circ (A_i\sharp_{1-s} B_i)+(A_j\sharp_{1-s} B_j)\circ (A_i\sharp_{s}B_i)\nonumber\\&
\,+\left(\frac{t-s}{s-1/2}\right)
\left((A_j\sharp_{s} B_j)\circ (A_i\sharp_{1-s} B_i)+(A_j\sharp_{1-s} B_j)\circ (A_i\sharp_{s} B_i)-2(A_j\sharp B_j)\circ (A_i\sharp B_i)\right)\Big)\\&\leq \frac{1}{2}\sum_{i,j=1}^n\left((A_j\sharp_{t}B_j)\circ (A_i\sharp_{1-t} B_i)+(A_j\sharp_{1-t}B_j)\circ (A_i\sharp_{t}B_i)\right)\qquad(\textrm{by inequality\,\eqref{9090}})\\&
=\sum_{j=1}^n(A_j\sharp_{t}B_j)\circ \sum_{j=1}^n(A_j\sharp_{1-t} B_j)\,.
\end{align*}
\end{proof}
If we put $B_j=I\,\,(1\leq j\leq n)$ in Theorem \ref{okmn}, then we get the next result.

\begin{corollary}
Let $A_j\in{\mathbb B}({\mathscr H})_+\,\,(1\leq j\leq n)$. Then
\begin{align*}
\left(\sum_{j=1}^nA_j^{s}\right)&\circ \left(\sum_{j=1}^nA_j^{1-s}\right)
\\&+\left(\frac{t-s}{s-1/2}\right)\left(\left(\sum_{j=1}^nA_j^{s}\right)\circ \left(\sum_{j=1}^nA_j^{1-s}\right)
-\left(\sum_{j=1}^nA_j^{\frac{1}{2}}\right)\circ \left(\sum_{j=1}^nA_j^{\frac{1}{2}}\right)\right)
\\&\leq \left(\sum_{j=1}^nA_j^{t}\right)\circ \left(\sum_{j=1}^nA_j^{1-t}\right)\,,
 \end{align*}
 where $1\geq t\geq s>{\frac{1}{2}}$ or $0\leq t\leq s<\frac{1}{2}$.
\end{corollary}

\bigskip

\bibliographystyle{amsplain}

\begin{thebibliography}{99}

\bibitem{ABM} Lj. Arambasi\'c, D. Baki\'c and M.S. Moslehian, \textit{A treatment of the Cauchy--Schwarz inequality in $C^*$-modules}, J. Math. Anal. Appl. \textbf{381} (2011) 546--556.

\bibitem{mo-mo1} M. Bakherad and M.S. Moslehian, \textit{Reverses and variations of Heinz inequality}, Linear Multilinear Algebra (to appear).

\bibitem{Bour} J.C. Bourin, E.Y. Lee, M. Fujii and Y. Seo,  \textit{A matrix reverse H\"{o}lder inequality}. Linear Algebra Appl. \textbf{431} (2009), no. 11, 2154--2159.

\bibitem{CAL} D.K. Callebaut, \textit{Generalization of the Cauchy--Schwarz inequality}, J. Math. Anal. Appl. \textbf{12} (1965), 491–-494.

\bibitem{eli} D.E. Daykin, C.J. Eliezer and C. Carlitz, Problem 5563, Amer. Math. Monthly \textbf{75} (1968), 198 and \textbf{76} (1969), 98--100.

\bibitem{abc} T. Furuta, J. Mi\'{c}i\'{c} Hot, J. Pe\v{c}ari\'{c} and Y. Seo,
\textit{Mond Pe\v{c}ari\'{c} method in operator inequalities}, Zagreb, 2005.

\bibitem{abc2} M. Fujii, J. Mi\'ci\'c Hot, J. Pe\v{c}ari\'{c} and Y. Seo, \textit{Recent developments of Mond--Pe\v{c}ari\'{c} method in operator inequalities. Inequalities for bounded selfadjoint operators on a Hilbert space. II.}, Monographs in Inequalities 4. Zagreb: Element, 2012.

\bibitem{I-V} D. Ili\v{s}evi\'c and S. Varo\v{s}anec, \textit{On the Cauchy--Schwarz inequality and its reverse in
semi-inner product $C^*$-modules}, Banach J. Math. Anal. \textbf{1} (2007), 78--84.

\bibitem{HZ} F. Hiai and X. Zhan, \textit{Inequalities involving unitarily invariant norms and operator monotone functions}, Linear Algebra Appl. \textbf{341} (2002), 151--169.

\bibitem{ando} F. Kubo and T. Ando, \textit{Means of positive linear operators}, Math. Ann. \textbf{246} (1980) 205–-224.

\bibitem{LEE} E.Y. Lee, \textit{A matrix reverse Cauchy-Schwarz inequality}. Linear Algebra Appl. \textbf{430} (2009), no. 2-3, 805--810.

\bibitem{MT} A. Matsumoto and M. Tominaga, \textit{Mond--Pe\v{c}ari\'{c} method for a mean-like transformation of operator functions}, Sci. Math. Jpn. \textbf{61} (2005), no. 2, 243–-247.

\bibitem{caleba} M.S. Moslehian, J.S. Matharu and J.S. Aujla, \textit{Non-commutative Callebaut inequality},
Linear Algebra Appl. \textbf{436} (2012) no. 9, 3347--3353.

\bibitem{paul}V.I. Paulsen, \textit{Completely bounded maps and dilation}, Pitman Research Notes in Mathematics Series, \textbf{146}, John Wiley $\&$ Sons, Inc., New York, 1986.

\bibitem{wada} S. Wada, \textit{On some refinement of the Cauchy--Schwarz inequality}, Linear Algebra Appl. \textbf{420} (2007) no. 2-3, 433--440.

\end{thebibliography}

\end{document}